\documentclass[reqno]{amsart}
\usepackage{amsmath,amsthm,amscd,amssymb,amsfonts, amsbsy}
\usepackage{latexsym, color, enumerate}
\usepackage{pxfonts}
\usepackage{mathrsfs}
\usepackage{marginnote}
\usepackage{todonotes}
\usepackage{hyperref}

\usepackage{tikz}%draw picture
\usetikzlibrary{snakes}%zigzagline
\usepackage{multicol}%multicolumn

\numberwithin{equation}{section}

\theoremstyle{plain}
\newtheorem{theorem}{Theorem}[section]
\newtheorem{lemma}[theorem]{Lemma}
\newtheorem{corollary}[theorem]{Corollary}
\newtheorem{proposition}[theorem]{Proposition}

\theoremstyle{definition}

\newtheorem{assumption}[theorem]{Assumption}

\theoremstyle{remark}
\newtheorem{remark}[theorem]{Remark}
\newtheorem{notation}[theorem]{Notation}

\newcommand{\bH}{\mathbb{H}}
\newcommand{\bI}{\mathbb{I}}

\newcommand{\bQ}{\mathbb{Q}}
\newcommand{\bR}{\mathbb{R}}

\newcommand\cB{\mathcal{B}}

\newcommand\cD{\mathcal{D}}

\newcommand\cH{\mathcal{H}}

\newcommand\cN{\mathcal{N}}

\newcommand\cP{\mathcal{P}}
\newcommand\cQ{\mathcal{Q}}

\newcommand\cM{\mathcal{M}}

\providecommand{\norm}[1]{\lVert#1\rVert}

\providecommand{\bignorm}[1]{\bigl\lVert#1\bigr\rVert}

\renewcommand{\vec}[1]{\boldsymbol{#1}}

\def\Xint#1{\mathchoice
	{\XXint\displaystyle\textstyle{#1}}%
	{\XXint\textstyle\scriptstyle{#1}}%
	{\XXint\scriptstyle\scriptscriptstyle{#1}}%
	{\XXint\scriptscriptstyle\scriptscriptstyle{#1}}%
	\!\int}
\def\XXint#1#2#3{{\setbox0=\hbox{$#1{#2#3}{\int}$}
		\vcenter{\hbox{$#2#3$}}\kern-.5\wd0}}
\def\dashint{\Xint-}%For the average integral symbol

\newcommand{\p}{\partial}

\begin{document}

\title[Mixed boundary value problem]{Mixed boundary value problems for parabolic equations in Sobolev spaces with mixed-norms} %with time-dependent interfacial boundary}

\author[J. Choi]{Jongkeun Choi}
\address[J. Choi]{Department of Mathematics Education, Pusan National University, Busan 46241, Republic of Korea}

\email{jongkeun\_choi@pusan.ac.kr}

\thanks{
J. Choi was partially supported by the National Research Foundation of Korea (NRF) under agreements  NRF-2019R1F1A1058826 and NRF-2022R1F1A1074461}

\author[H. Dong]{Hongjie Dong}	
\address[H. Dong]{Division of Applied Mathematics, Brown University, 182 George Street, Providence, RI 02912, USA}
\email{hongjie\_dong@brown.edu}
\thanks{H. Dong was partially supported by the Simons Foundation, grant no. 709545, a Simons fellowship, grant no. 007638, and the NSF under agreement DMS-2055244.}

\author[Z. Li]{Zongyuan Li}
\address[Z. Li]{Department of Mathematics, Rutgers University, 110 Frelinghuysen Road, Piscataway, NJ 08854-8019, USA}
\email{zongyuan.li@rutgers.edu}
\thanks{Z. Li was partially supported by an AMS-Simons travel grant.}


\begin{abstract}
We establish $L_{q,p}$-estimates and solvability for mixed Dirichlet-conormal problems for parabolic equations in a cylindrical Reifenberg-flat domain with a rough time-dependent separation.
\end{abstract}

\maketitle
%\tableofcontents

%%========================================
\section{Introduction}
%%========================================

Let $\cQ^T$ be a cylindrical domain in $\bR^{d+1}$ of the form
$$
\cQ^T=(-\infty, T)\times \Omega,
$$
where $T\in (-\infty, \infty]$ and $\Omega$ is either a bounded or unbounded domain in $\bR^d$, $d \geq 2$.
The lateral boundary of $\cQ^T$ is divided into two components $\cD^T$ and $\cN^T$ separated by $\Gamma^T$, which is allowed to be time-dependent.
See Figure \ref{fig1} below.
We consider mixed boundary value problems for parabolic equations
\begin{equation}		\label{211231@eq1}
	\begin{cases}
		\cP u -\lambda u= D_ig_i+f  & \text{in }\, \cQ^T,\\
		\cB u= g_in_i& \text{on }\, \cN^T,\\
		u = 0 & \text{on }\, \cD^T,
	\end{cases}
\end{equation}
where the operators $\cP$ and $\cB$ are defined by
$$
\cP u=- u_t+D_i (a^{ij}D_j u), \quad \cB u=a^{ij}D_j u n_i,
$$
and $n=(n_1,\ldots, n_d)$ is the outward unit normal to $\partial \Omega$.
The leading coefficients $a^{ij}$ are assumed to be symmetric and satisfy the uniform ellipticity condition.
The boundary value problem \eqref{211231@eq1} arises naturally in mathematical physics and material science dealing with  metallurgical melting, combustion, and wave phenomena, etc.
We refer the reader to \cite{MR0216018, MR647124, MR719445, MR1175754,MR1804512,MR1799414, MR2321139} and references therein.
It is also partly motivated by modeling  exocytosis, which have a form of active transport mechanism. See \cite{MR3190265}.

\begin{figure}[b]		\label{fig1}
\begin{tikzpicture}
	\draw [->] (-2,-3.3) -- (-2, 0) node (taxis) [right] {$t$};
	\draw [->] (-2,-3.3) -- (2,-3.3) node (xaxis) [above] {$x$};
	\draw (0,-0.2) ellipse (1.5 and 0.2);%topcircle
	\draw (-1.5,-0.2) -- (-1.5,-3);%leftend
	\draw[snake=zigzag, red] (0,-0.4) -- (0,-3.2);%Gamma
	\draw (-1.5,-3) arc (180:360:1.5 and 0.2);%bottom lower half circle
	\draw [dashed] (-1.5,-3) arc (180:360:1.5 and -0.2);%bottom upper half circle
	\draw (1.5,-0.2) -- (1.5,-3);  %rightend
	\fill [gray,opacity=0.3] (-1.5,-0.2) -- (-1.5,-3) arc (180:360:1.5 and 0.2) -- (1.5,-0.2) arc (0:180:1.5 and -0.2);%fill front surface
	\node[left] at (0,-0.6) {$\Gamma$};
	\node at (1,-1.8) {$\cD$};
	\node at (-0.8,-1.8) {$\cN$};
	\node at (1.8,-0.2) {$\cQ$};
%	\node at (1.2,-0.6) {$\p_l\cQ$};
\end{tikzpicture}
\caption{Cylindrical domain with time-dependent separation}
\end{figure}
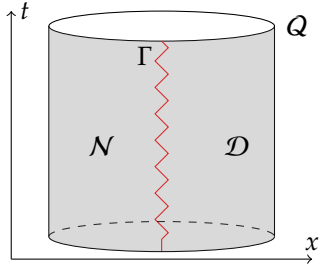

In a recent paper \cite{MR4387198}, we proved the unique  solvability in unmixed-norm Sobolev spaces $\cH^1_p$ (see \eqref{eqn-220206-0348} and \eqref{eqn-220206-0353}) for the problem \eqref{211231@eq1} when the coefficients $a^{ij}$ have small bounded mean oscillation (BMO) with respect to all the variables $(t,x)$, the base domain $\Omega$ is Reifenberg-flat, and the separation $\Gamma$ is locally close to the graph of a Lipschitz function of $m$ variables, where
$$
m\in \{0,1,\ldots, d-2\} \quad \text{and} \quad p\in \bigg(\frac{2(m+2)}{m+3}, \frac{2(m+2)}{m+1}\bigg).
$$
For precise conditions on the domain and separation, see Assumption \ref{A11}.
Notice that if $\Gamma$ is Reifenberg-flat of codimension $2$ (i.e., $m=0$),
such range $p\in (4/3, 4)$ is optimal even in the stationary case, in view of the following classical example
\begin{equation*}
	u(x,y)= \operatorname{Im}\,(x+iy)^{1/2}
\end{equation*}
which is harmonic on the upper half space $\{y>0\}$, and
\begin{equation*}
	u=0\,\,\text{on}\,\,\{x>0, y=0\},\quad \p u/\p\vec{n}=0 \,\,\text{on}\,\,\{x<0, y=0\}.
\end{equation*}
For other previous results on the mixed boundary value problems for parabolic equations in unmixed-norm spaces, we refer the reader to \cite{MR322377, MR678501, MR3804727} and the references therein.
We note that in these work, either $p$ is assumed to be $2$ or an implicit condition is imposed on the operator so that $p$ needs to be sufficiently close to $2$.

It is worth mentioning that here we focus on the case when $d \geq 2$ and the two types of boundary conditions touch at an nonempty $\Gamma^T$. This is the case when the difficulty of ``optimal regularity'' appears. 
When $\overline{\cD}\cap \overline{\cN}=\emptyset$, we can just apply the estimates for the pure Dirichlet and pure conormal problems and a partition of unity argument. In this case, the solutions are smooth given that the domain, the operator, and the boundary data are all smooth. In particular, when $d=1$, $\Omega$ is actually an interval $(a,b)$. The mixed problems just mean that different boundary conditions are assigned at $x= a$ and $x=b$.

In this paper, we extend the result in \cite{MR4387198}  by proving $L_q^t(L_p^x)$ mixed-norm estimates for $Du$ and the $\cH^1_{q,p}$ (see \eqref{eqn-220206-0348}) solvability when
$$
p\in \bigg(\frac{2(m+2)}{m+3}, \frac{2(m+2)}{m+1}\bigg) \quad \text{and} \quad q\in \bigg[p, \frac{2p}{(m+1)(p-2)_+}\bigg),
$$
under the same smoothness assumptions on the coefficients,  domain,  and  separation.
In the special case when $\Gamma$ is time-independent, we get the solvability for all
\begin{equation} \label{eqn-220714-0115}
	p\in \bigg(\frac{2(m+2)}{m+3}, \frac{2(m+2)}{m+1}\bigg) \quad \text{and} \quad q\in (1,\infty).
\end{equation}
In particular, when $\p\Omega$ and $\Gamma^T$ are smooth enough, we can make a change of variables to locally flatten $\p\Omega$ first and then make $\Gamma^T$ to be time-independent. Hence, the full solvability range in \eqref{eqn-220714-0115} is achieved.

In \cite{MR1444765} Savar\'{e} considered parabolic equations in a cylindrical domain with $C^{1,1}$ base domain and separation.
Under a uniform linear bound condition on the excess of the separation with respect to $t$, he proved the unique solvability in 
$L_2$-based Sobolev spaces.
We also refer the reader to Hieber-Rehberg \cite{MR2403322} for quasilinear parabolic systems of reaction-diffusion equations in a cylindrical domain with Lipschitz base domain in $\bR^d$ for $d=2, \, 3$ and time-independent separation.
Assuming an implicit topological isomorphism condition on the second-order operator, they established the solvability in mixed-norm spaces with $p=2$ and $q>c$ for some $c$ depending on the operator and dimension.
To the best of the authors' knowledge, our results regarding the mixed boundary value problem are the first to deal with mixed-norm estimates for $p\neq 2$ 
 even  in the case when $\partial \Omega$ and $\Gamma$ are smooth and $\Gamma$ is time-independent.
For other previous results on mixed-norm estimates for purely Dirichlet or conormal derivative boundary value problems, we refer the reader to \cite{MR2352490, MR2561181, MR3231530, MR3630407, MR3812104, MR3947859, MR4387945} and references therein.

The proof in \cite{MR4387198} relies on a decomposition argument using a carefully designed cut-off function near the separation $\Gamma$ and on a  level set method with  the measure theoretic ``crawling of ink spots" lemma originally due to Krylov and Safonov \cite{MR0579490,MR0563790}.
While in \cite{MR4387198} we used the decomposition argument and estimates in $L_2$-based spaces, in this paper, to prove our main result, we refine the decomposition argument in \cite{MR4387198} in the setting of $L_p$-based spaces for $p<2$ and exploit an idea of Krylov \cite{MR2352490} to utilize the level set method in the $t$-variable only.
Because the decomposition argument fails for $p>2$, it remains open whether the mixed-norm estimates hold for
$$
p\in \bigg(2, \frac{2(m+2)}{m+1}\bigg), \quad q\in \bigg(\frac{2p}{(m+1)(p-2)}, \infty\bigg)
$$
when $\Gamma$ is time-dependent.
See the explanation after \eqref{eqn-211220-0457}.

The remainder of the paper is organized as follows.
In the next section, we introduce some notation and state our main result of the paper.
In Section \ref{S3}, we derive certain local estimates, which are used in Section \ref{sec-211227-0648} for the level set argument in the $t$-variable.
Finally, we complete the proof of the main result in Section \ref{S5}.

%%========================================
\section{Notation and main result}\label{sec-notation}
%%========================================

We first introduce some notation used throughout the paper.
We use $X=(t,x)$ to denote a generic point in the Euclidean space $\bR^{d+1}$, where $d\ge 2$ and $x=(x^1,\ldots, x^d)\in \bR^d$.
We also write $Y=(s, y)$ and $X_0=(t_0, x_0)$, etc.
Let $\cQ$ be a cylindrical domain in $\bR^{d+1}$ of the form
$$
\cQ=(-\infty, \infty)\times \Omega,
$$
where $\Omega$ is a domain in $\bR^d$.
We assume that the lateral boundary of $\cQ$, denoted by $\partial \cQ=(-\infty, \infty)\times \partial \Omega$, is divided into two components $\cD$ and $\cN$ separated by $\Gamma$, i.e.,
as in  \cite{MR4387198},
$\cD\subset \partial \cQ$ is an open set (relative to $\partial \cQ$) and
$$
\cN=\partial \cQ\setminus \cD, \quad \Gamma=\overline{\cD}\cap \overline{\cN}.
$$
Note that the separation $\Gamma$ is allowed to be time-dependent.
For $T\in (-\infty, \infty]$, we define
$$
\cQ^T=\{X\in \cQ: t<T\}
$$
and similarly define $\cD^T$, $\cN^T$, and $\Gamma^T$.
For $R>0$, we denote the parabolic cylinders by
$$
\begin{aligned}
Q_R(X)&=(t-R^2, t)\times B_R(x), \\
\bQ_R(X)&=(t-R^2, t+R^2)\times B_R(x),\\
\cQ_R(X)&=\cQ\cap Q_R(X),
\end{aligned}
$$
where $B_R(x)$ is the usual Euclidean ball of radius $R$ centered at $x$.
The center will be omitted when it is the origin, i.e.,  for instance, we write $Q_R$ for $Q_R(0)$.

For a function $u$ on an open set $Q\subset \bR^{d+1}$, we set
$$
(u)_Q=\frac{1}{|Q|} \int_Q u\, dX=\dashint_{Q} u\,dX,
$$
where $|Q|$ is the $d+1$-dimensional Lebesgue measure of $Q$.
For $p,q\in [1, \infty)$, we define the mixed-norm on $Q$ by
$$
\|u\|_{L^t_qL^x_p(Q)}=\Bigg(\int_{\bR} \Bigg(\int_{\bR^d} |u|^p \bI_{Q}\,dx\Bigg)^{q/p}\,dt \Bigg)^{1/q},
$$
where $\bI_{Q}$ is the usual characteristic function.
Similarly, we define $L_q^tL_p^x$-norms with $p=\infty$ or $q=\infty$, and $L^x_p L^t_q$-norms.
We often write
$L_{q,p}$ and $L_p$ for $L_q^tL_p^x$ and $L_{p,p}$.

We set
$$
W^{0,1}_{q,p}(\cQ^T)=\{u: u\in L_{q,p}(\cQ^T),\, D_{x} u\in L_{q,p}(\cQ^T)^d\}
$$
with the norm
$$
\|u\|_{W^{0,1}_{q,p}(\cQ^T)}=\|u\|_{L_{q,p}(\cQ^T)}+\|Du\|_{L_{q,p}(\cQ^T)},
$$
and we denote by $W^{0,1}_{q,p, \cD^T}(\cQ^T)$ the closure of $C^\infty_{\cD^T} (\cQ^T)$ in $W^{0,1}_{q,p}(\cQ^T)$, where
$C^\infty_{\cD^T}(\cQ^T)$ is the set of all infinitely differentiable functions on $\bR^{d+1}$ having a compact support in $\overline{\cQ^T}$ and vanishing in a neighborhood of $\cD^T$.
We also set
\begin{equation}\label{eqn-220206-0348}
\cH_{q,p, \cD^T}^1 (\cQ^T)=\big\{u: u\in W^{0,1}_{q,p, \cD^T}(\cQ^T),\,  u_t\in \bH^{-1}_{q,p, \cD^T}(\cQ^T)\big\}
\end{equation}
with the norm
$$
\|u\|_{\cH^1_{q,p, \cD^T} (\cQ^T)}=\|u\|_{W^{0,1}_{q,p}(\cQ^T)}+\|u_t\|_{\bH^{-1}_{q,p,\cD^T}(\cQ^T)},
$$
where by $u_t\in \bH^{-1}_{q,p, \cD^T}(\cQ^T)$, we mean that there exist $g=(g_1,\ldots, g_d)\in L_{q,p}(\cQ^T)^d$ and $f\in L_{q,p}(\cQ^T)$ satisfying
$$
u_t=D_i g_i+f \quad \text{in }\, \cQ^T, \quad g_i n_i=0 \quad \text{on }\, \cN^T
$$
in the following distribution sense
\begin{equation*}
(u_t, \varphi):=\int_{\cQ^T} -u \varphi_t\,dX = \int_{\cQ^T} (-g_iD_i \varphi+f\varphi)\,dX
\end{equation*}
for all $\varphi \in C^\infty_{\cD^T}(\cQ^T)$ vanishing at $t=T$,
and that
$$
\|u_t\|_{\bH^{-1}_{q,p, \cD^T}(\cQ^T)}=\inf \big\{ \|g\|_{L_{q,p}(\cQ^T)}+\|f\|_{L_{q,p}(\cQ^T)}: u_t=D_i g_i+f\, \text{ in }\, \cQ^T,\ g_i n_i =0 \, \text{ on }\, \cN^T\big\}
$$
is finite.
We abbreviate
\begin{equation}\label{eqn-220206-0353}
	\cH^1_{p,p, \cD^T} (\cQ^T)=\cH^1_{p, \cD^T}(\cQ^T).
\end{equation}
Throughout this paper, we discuss weak solutions to the problem \eqref{211231@eq1}, which means the following integral identity holds for all $\varphi \in C^\infty_{\cD^T}(\cQ^T)$ vanishing at $t=T$,
\begin{equation}\label{eqn-220715-1259}
	\int_{\cQ^T} u\varphi_t \,dX+ \int_{\cQ^T} (-a^{ij}D_j u D_i \varphi - \lambda u \varphi )\,dX= \int_{\cQ^T} (-g_i D_i \varphi + f\varphi)\,dX.
\end{equation}
We also discuss ``local weak solutions'' as, for instance, in \eqref{eqn-220715-1258}, in which case, we mean that \eqref{eqn-220715-1259} holds with $f=0$ for any test function
$\varphi \in
C^\infty(\overline{\cQ_R})$ vanishing on $\p\cQ_R \setminus \cN$.

\subsection{Assumptions and main result}
Throughout this paper, we assume that the leading coefficients $a^{ij}$ of the operator $\cP$ are symmetric and satisfy the uniform ellipticity condition
$$
a^{ij}(X)\xi_j \xi_i \ge \Lambda |\xi|^2, \quad |a^{ij}(X)|\le \Lambda^{-1}
$$
for all $X\in \bR^{d+1}$, $\xi\in \bR^d$, and for some constant $\Lambda\in (0,1]$.
Regarding the symmetric condition, see Remark \ref{220719@rmk1} for an explanation.
We impose the following small BMO condition on the leading coefficients, where $\theta\in (0,1)$ is a parameter to be specified.

\begin{assumption}[$\theta$]\label{A2}
For any $X_0\in \overline{\cQ}$ and $R\in (0, R_0]$, we have
	$$
	\dashint_{\cQ_R(X_0)}|a^{ij}(X)-(a^{ij})_{\cQ_R(X_0)}|\,dX\le \theta.
	$$
\end{assumption}

We also impose the following regularity assumptions on the boundary of the domain and separation, where $\gamma\in (0,1)$ is a parameter to be specified.

\begin{assumption}[$\gamma; m, M$]		\label{A11}
	Let $m\in \{0,1,\ldots, d-2\}$ and $M\in (0, \infty)$.
	\begin{enumerate}[$(a)$]
		\item
		For any $x_0\in \partial \Omega$ and $R\in (0, R_0]$, there is a coordinate system depending on $x_0$ and $R$ such that in this coordinate system, we have
		\begin{equation}		\label{200429@eq1}
			\{y: y^1>x_0^1+\gamma R\}\cap B_R(x_0)\subset \Omega \cap B_R(x_0)\subset \{y: y^1>x_0^1-\gamma R\}\cap B_R(x_0).
		\end{equation}
		
		\item
		For any $X_0=(t_0,x_0)\in \Gamma$ and $R\in (0, R_0]$,   there exist
		a spatial coordinate system and a Lipschitz function $\phi$ of $m$ variables with Lipschitz constant $M$,  such that in the new coordinate system, we have  \eqref{200429@eq1},
		$$
		\big(\partial \cQ\cap \bQ_R(X_0) \cap \{(s,y): y^2>\phi(y^3,\ldots,y^{m+2})+\gamma R\} \big)\subset \cD,
		$$
		$$
		\big(\partial \cQ \cap \bQ_R(X_0)\cap  \{(s,y): y^2<\phi(y^3,\ldots, y^{m+2})-\gamma R\}\big)\subset \cN,
		$$
		and
		$$
		\phi(x_0^3,\ldots, x_0^{m+2})=x_0^2.
		$$
		Here, if $m=0$, then the function $\phi$ is  understood as the constant function $\phi\equiv x_0^2$.
	\end{enumerate}
\end{assumption}

Noting that Assumption \ref{A11} $(\gamma;0,M)$ holds when
$\p\Omega$ is locally given by the graph $\{x^1=\psi(x^2,\ldots,x^d)\}$ and $\Gamma$ is locally given by its intersection with $\{x^2=\tilde{\psi}(t,x^1,x^3,\ldots,x^d)\}$, where $\psi$ and $\tilde{\psi}$ are Lipschitz functions ($\tilde{\psi}$ in the parabolic metric) with correspondingly small constants.
The assumption also includes certain fractal structures.

The main result in the current paper reads as follows.

\begin{theorem}\label{thm-201022-0843}
Let
$R_0\in (0,1]$,  $m\in \{0,1,\ldots,d-2\}$, $M\in (0, \infty)$, and let
\begin{equation}		\label{220203@eq1a}
p\in \bigg(\frac{2(m+2)}{m+3}, \frac{2(m+2)}{m+1}\bigg), \quad q\in \bigg(p, \frac{2p}{(m+1)(p-2)_+}\bigg).
\end{equation}
There exist constants $\theta, \gamma\in (0,1)$ and $\lambda_0\in (0, \infty)$ with
$$
(\theta, \gamma) = (\theta, \gamma)(d,\Lambda,M,p,q), \quad \lambda_0=\lambda_0(d,\Lambda, M, p,q, R_0),
$$
such that if Assumptions \ref{A2} $(\theta)$ and \ref{A11} $(\gamma;m,M)$ are satisfied, then we have the following.
For any $\lambda\ge \lambda_0$, $g=(g_1,\ldots, g_d)\in L_{q,p}(\cQ^T)^d$, and $f\in L_{q,p}(\cQ^T)$, there exists a unique solution $u\in \cH^1_{q,p, \cD^T}(\cQ^T)$ to
\begin{equation}		\label{220203@eq1}
	\begin{cases}
		\cP u -\lambda u= D_ig_i+f  & \text{in }\, \cQ^T,\\
		\cB u= g_in_i& \text{on }\, \cN^T,\\
		u = 0 & \text{on }\, \cD^T
	\end{cases}
\end{equation}
satisfying
\begin{equation}\label{eqn-201028-0521}		\norm{Du}_{L_{q,p}(\cQ^T)}+\lambda^{1/2}\norm{u}_{L_{q,p}(\cQ^T)} \le C \norm{g}_{L_{q,p}(\cQ^T)}+C\lambda^{-1/2}\norm{f}_{L_{q,p}(\cQ^T)},
	\end{equation}
where $C=C(d, \Lambda, M, p, q)$.
The same result holds for any $p,q$ satisfying
$$
p\in \bigg(\frac{2(m+2)}{m+3}, \frac{2(m+2)}{m+1}\bigg), \quad q\in (1, \infty)
$$
instead of \eqref{220203@eq1a} when $\Gamma$ is time-independent.
\end{theorem}

In Figure \ref{fig:1/p-1/q}, we draw a diagram to show the range of $(p,q)$ in \eqref{220203@eq1a}.
\begin{figure}[h]		%\label{fig2}
\begin{tikzpicture}[scale=1]
	\draw[fill=gray!30] (2,0) -- (3,0) -- (3,3) -- (1,1) -- cycle;\draw [<->,thick] (0,4.5) node (yaxis) [above] {$1/q$} |- (5,0) node (xaxis) [right] {$1/p$};
	\draw (4.3,3.4) node {$q=p$};
	\draw (0,0)  -- (4,4);
	\draw[dashed] (1,1) -- (1,0) node[below] {$\frac{m+1}{2(m+2)}$};
	\draw[dashed] (2,2) -- (2,0) node[below] {$\frac{1}{2}$};
	\draw (3,3) -- (3,0) node[below] {$\frac{m+3}{2(m+2)}$};
	\draw (1,1) -- (2,0);
\end{tikzpicture}
\caption{Range of $(p,q)$} \label{fig:1/p-1/q}
\end{figure}
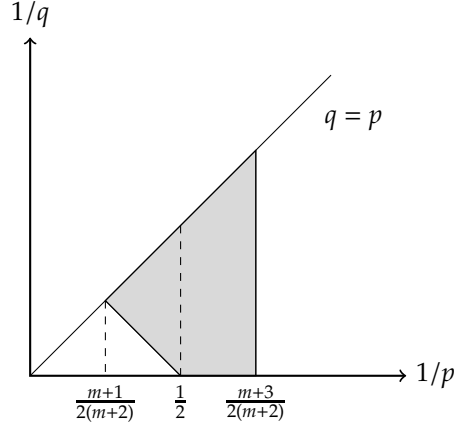
\begin{remark}
It is not clear to us if Theorem \ref{thm-201022-0843} still holds when $p\in \big(2, \frac{2(m+2)}{m+1}\big)$, $q\in \big(\frac {2p}{(m+1)(p-2)},\infty\big)$, and $\Gamma$ is time-dependent. In fact, the decomposition argument in Section \ref{S6} fails if $p>2$.
\end{remark}
\begin{remark}
	\begin{enumerate}[$(a)$]
		\item
Based on the method of continuity, one can easily extend the results in Theorem \ref{thm-201022-0843} to parabolic operators with bounded lower-order terms
$$
\begin{aligned}
\cP u&=- u_t+D_i (a^{ij}D_j u+a^i u)+b^i D_i u+cu,\\
\cB u&= (a^{ij}D_j u+a^i u)n_i,
\end{aligned}
$$
at the cost of possibly increasing the constant $\lambda_0$.

		\item From Theorem \ref{thm-201022-0843}, we can also obtain the solvability of the initial boundary value problem on $(0,T)\times \Omega$ with the zero initial condition. In this case, we can take $\lambda_0=0$ with the help of the standard trick of considering $e^{-\lambda_0 t}u$, cf. \cite[Theorem 8.2 (iii)]{MR2835999}.
	\end{enumerate}
\end{remark}

\begin{remark}		\label{220719@rmk1}
In theorem \ref{thm-201022-0843} and throughout the paper, we require the symmetry of the coefficients $a^{ij}$ for the optimal range $p\in (4/3, 4)$ when $m=0$ for mixed boundary value problems.
We refer the reader to \cite[Theorem 4.1]{MR4261267} and \cite[Proposition~4.4]{MR4387198} for the optimal estimates for model problems - Laplace and heat equations with flat boundary and separation.
Notice that if $a^{ij}$ is not symmetric, then the range of $p$ can be more restrictive. See \cite[Example 2.8]{MR4261267}.
\end{remark}

%%========================================
\section{Preliminary estimates}		\label{S3}
%%========================================

Hereafter in this paper, we use the following notation.
\begin{notation}
For nonnegative (variable) quantities $A$ and $B$, we denote $A \lesssim B$ if there exists a generic positive constant $C$ such that $A \le CB$.
We add subscript letters like $A\lesssim_{a,b} B$ to indicate the dependence of the implicit constant $C$ on the parameters $a$ and $b$.
\end{notation}

\begin{notation}
For a given constant $\lambda\ge 0$ and functions $u$, $f$, and $g=(g_1,\ldots, g_d)$, we write
$$
U=|Du|+\lambda^{1/2} |u|, \quad F=|g|+\lambda^{-1/2} |f|,
$$
where we take $f=0$ and $F=|g|$ whenever $\lambda=0$.
\end{notation}
The following constants in Assumptions \ref{A2} and \ref{A11} are fixed throughout the paper:
$$
	R_0\in (0,1],  \quad m\in \{0,1,\ldots,d-2\}, \quad M\in (0, \infty).
$$

%%========================================
\subsection{\texorpdfstring{$\cH^1_p$}{H1p} solvability and localization}
%%========================================

In this section, we derive some local estimates, in the proofs of which, we shall use the unmixed-norm $L_p$-estimates proved in  \cite{MR4387198}.
For the reader’s convenience, we present here the main result in \cite{MR4387198}.

\begin{theorem}[{\cite[Theorem 2.4]{MR4387198}}]		\label{MT1}
Let $p\in \bigg(\frac{2(m+2)}{m+3}, \frac{2(m+2)}{m+1}\bigg)$.
There exist constants $\theta, \gamma\in (0,1)$ and $\lambda_0\in (0, \infty)$ with
$$
(\theta, \gamma)=(\theta, \gamma)(d, \Lambda, M, p), \quad \lambda_0=\lambda_0(d,\Lambda, M, p, R_0),
$$
such that if Assumptions \ref{A2} $(\theta)$ and \ref{A11} $(\gamma;m,M)$ are satisfied, then
the following assertions hold.
\begin{enumerate}[$(a)$]
\item
For any $\lambda\ge \lambda_0$, $g=(g_1,\ldots, g_d)\in L_{p}(\cQ^T)^d$, and $f\in L_{p}(\cQ^T)$, there exists a unique solution $u\in \cH^1_{p, \cD^T}(\cQ^T)$ to \eqref{220203@eq1}, which satisfies
$$
\|U\|_{L_p(\cQ^T)}\lesssim_{d, \Lambda, M, p} \|F\|_{L_p(\cQ^T)}.
$$
\item
Let $T\in (0, \infty)$.
For any $g=(g_1,\ldots, g_d)\in L_{p}(\tilde{\cQ})^d$ and $f\in L_{p}(\tilde{\cQ})$, there exists a unique solution $u\in \cH^1_{p, \tilde{\cD}}(\tilde{\cQ})$ to the initial boundary value problem
$$
	\begin{cases}
		\cP u = D_ig_i+f  & \text{in }\, \tilde{\cQ}:=(0, T)\times \Omega,\\
		\cB u= g_in_i& \text{on }\, \tilde{\cN}:=((0, T)\times \partial \Omega)\cap \cN,\\
		u = 0 & \text{on }\, \tilde{\cD}:=((0, T)\times \partial \Omega)\cap \cD,\\
		u=0 & \text{on }\, \{0\}\times \Omega,
	\end{cases}
$$
which satisfies
\begin{equation}		\label{220725@eq1}
\|u\|_{\cH^1_{p, \tilde{\cD}}(\tilde{\cQ})}\lesssim_{d, \Lambda, M, p, R_0, T} \|g\|_{L_p(\tilde{\cQ})}+\|f\|_{L_{p}(\tilde{\cQ})}.
\end{equation}
\end{enumerate}
\end{theorem}

\begin{remark}		\label{220724@rmk1}
Theorem \ref{MT1} $(b)$ still holds with $f\in L_q(\tilde{\cQ})$ and
$$
\|u\|_{\cH^1_{p, \tilde{\cD}}(\tilde{\cQ})}\lesssim_{d, \Lambda, M, p, R_0, T} \|g\|_{L_p(\tilde{\cQ})}+\|f\|_{L_{q}(\tilde{\cQ})}
$$
instead of \eqref{220725@eq1}, if $q\in (1, p]$ is such that
$$
\frac{d+2}{q}\le 1+\frac{d+2}{p}.
$$
If the inequality above is strict, then we can take $q=1$.
The proof is based on a duality argument  combined with embedding result for parabolic Sobolev spaces  and  the standard approximation argument.
We refer the reader to \cite[Lemma 3.1]{MR4387945} for the duality argument and \cite[Theorem 5.2]{MR4387945} for the embedding result (with mixed-norms).
Notice that in \cite[Theorem 5.2]{MR4387945},  the boundedness of the domain is required.
However, if unmixed-norms are considered, then by a covering argument we can remove the boundedness condition so that the result is applicable to our case.
\end{remark}

From Theorem \ref{MT1}, we can obtain the following local estimates.
\begin{lemma}\label{lem-211217-0940}
Let $p\in \bigg(\frac{2(m+2)}{m+3}, \frac{2(m+2)}{m+1}\bigg)$ and $\theta,\gamma$ be the constants from Theorem \ref{MT1}.
If Assumptions  \ref{A2} $(\theta)$ and \ref{A11} $(\gamma;m,M)$ are satisfied with these $\theta$ and $\gamma$, then the following assertions hold.
Let $(0,0)\in \overline{\cQ}$, $R\in (0, R_0]$, and $u\in \cH^1_p(\cQ_R)$ satisfy
\begin{equation}\label{eqn-220715-1258}
		\begin{cases}
		\cP u  -\lambda u = D_i g_i&\quad \text{in }\, \cQ_R,\\
		\cB u=g_in_i &\quad \text{on }\, Q_R\cap\cN,\\
		u=0 &\quad \text{on }\, Q_R\cap\cD,
	\end{cases}
\end{equation}
where $g=(g_1,\ldots, g_d)\in L_p(\cQ_R)^d$.
\begin{enumerate}[$(a)$]
	\item When $\lambda =0$, we have that
	\begin{equation*}
		\|Du\|_{L_p(\cQ_{R/2})} \lesssim_{d,\Lambda, M, p} R^{-1}\|u\|_{L_p(\cQ_R)} + \|g\|_{L_p(\cQ_R)}.
	\end{equation*}
\item When $\lambda \geq 0$ and $g=0$, for any $\hat p<2(m+2)/(m+1)$, we have $U\in L_{\hat{p}}{(\cQ_{R/2})}$
with
\begin{equation}\label{eqn-200813-0849}
	(U^{\hat{p}})^{1/\hat p}_{\cQ_{R/2}} \lesssim_{d, \Lambda, M, p, \hat{p}} (U)_{\cQ_{R}}.
\end{equation}
\end{enumerate}
\end{lemma}

\begin{proof}
The result in $(a)$ is obtained by localizing the estimate in Theorem \ref{MT1} $(a)$. The proof is the same as that of \cite[Lemma 3.10]{MR4387198}, and hence is omitted.

To prove $(b)$, we first deal with the case $\lambda=0$.
By a standard bootstrap argument with the solvability result in Remark \ref{220724@rmk1}, we see that $Du\in L_{\hat{p}}(\cQ_\rho)$ for any $\hat{p}<2(m+2)/(m+1)$ and $\rho\in (0, R)$.
For the estimate \eqref{eqn-200813-0849} (with $\lambda=0$), we prove that
\begin{equation}
            \label{eqn-200919-0952}
		(|Du|^{\hat{p}})^{1/\hat{p}}_{\cQ_{r/4}(X_0)} \lesssim (|Du|^{p_*})^{1/{p_*}}_{\cQ_{r}(X_0)}
	\end{equation}
for $\hat{p}>\max\{p, 2\}$ and $p_* \in ((d+2)\hat{p}/(d+2+\hat{p}),\hat{p})$.
Here $\cQ_{r}(X_0)\subset\cQ_{R/2}$
and we have one of the following four cases: $Q_{r}(X_0)\subset \cQ$ (interior), or $X_0\in\p\cQ$ with $Q_{r}(X_0)\cap\p\cQ \subset \cD$ (Dirichlet), or $X_0\in\p\cQ$ with $Q_{r}(X_0)\cap\p\cQ \subset \cN$ (conormal), or $X_0\in\Gamma$ (mixed).
	
	{\it The Dirichlet or mixed cases.} From the result in $(a)$ with $p=\hat{p}$, we have
	\begin{equation*}%\label{eqn-200820-0441}
		(|Du|^{\hat{p}})^{1/\hat{p}}_{\cQ_{r/4}(X_0)}\lesssim r^{-1}(|u|^{\hat{p}})^{1/\hat{p}}_{\cQ_{r/2}(X_0)}.
	\end{equation*}
Then \eqref{eqn-200919-0952} can be obtained by applying the following Sobolev-Poincar\'e inequality in \cite[(3.9) in Lemma~3.8]{MR4387198} with $p_0=q_0=\hat{p}$ and $p=q=p_*$:
	\begin{equation*}%\label{eqn-211217-0948}
		(|u|^{\hat{p}})^{1/\hat{p}}_{\cQ_{r/2}(X_0)}\lesssim r (|Du|^{p_*})^{1/p_*}_{\cQ_{r}(X_0)}.
	\end{equation*}	
	
	{\it The conormal or interior cases.} Note that in either cases, if $u$ is a solution and $c$ is a constant, then $u-c$ is also a solution. Hence, the result in $(a)$ with $p=\hat{p}$ and $u$ being replaced with $u-(u)_{\cQ_{r/2}(X_0)}$ yields
	\begin{equation*}
		(|Du|^{\hat{p}})^{1/\hat{p}}_{\cQ_{r/4}(X_0)} \lesssim r^{-1}(|u-(u)_{\cQ_{r/2}(X_0)}|^{\hat{p}})^{1/\hat{p}}_{\cQ_{r/2}(X_0)}.
	\end{equation*}
	Then \eqref{eqn-200919-0952} can be obtained by applying the embedding \cite[(3.8) in Lemma~3.8]{MR4387198}.
	
	From \eqref{eqn-200919-0952}, the desired estimate \eqref{eqn-200813-0849} with $\lambda=0$ can be proved in a standard way: rescaling, covering, and iteration.
	
	The case when $\lambda>0$ can be proved by using Agmon's idea of considering $u(t,x)\cos(\sqrt{\lambda}y + \pi/4)$ with an artificial variable $y\in\mathbb{R}$, noting
	\begin{equation*}
		(\mathcal{P} + \p_{yy})(u\cos(\sqrt{\lambda}y + \pi/4)) = (\mathcal{P}-\lambda)(u \cos(\sqrt{\lambda}y + \pi/4)),
	\end{equation*}
	cf. the proof of \cite[Lemma~3.12]{MR4387198}.
	The lemma is proved
\end{proof}

%========================================
\subsection{Equations with constant coefficients and a time-independent separation}%\label{sec-211202-1117}
%========================================
In this section, we deal with
\begin{equation}\label{eqn-211217-1002}
	\begin{cases}
		-u_t + D_i(a^{ij}_0 D_ju) - \lambda u=0 &\quad \text{in }\, \cQ_R,\\
		a^{ij}_0 D_j u n_i = 0 &\quad \text{on }\, Q_R\cap\cN,\\
		u=0 &\quad \text{on }\, Q_R\cap\cD,
	\end{cases}
\end{equation}
where $(a^{ij}_0)_{i,j}$ is a constant symmetric matrix with the elliptic constant $\Lambda$ and the interfacial boundary
$$\Gamma\cap Q_R = \overline{\cD}\cap\overline{\cN}\cap Q_R$$
is time-independent.
In such a situation, we can differentiate both the equation and the boundary conditions in $t$. Furthermore, the usual time-average technique (the Steklov average) is available to build test functions.
\begin{lemma}\label{lem-200913-1026}
Let $p\in \bigg(\frac{2(m+2)}{m+3}, \frac{2(m+2)}{m+1}\bigg)$, $(0, 0)\in \Gamma$, and $R\in (0, R_0]$.
 For the constant
	$\gamma\in (0,1)$ from  Theorem \ref{MT1}, if Assumption \ref{A11} $(\gamma;m,M)$ is satisfied, then for any solution $u\in \cH^1_p(\cQ_R)$ to \eqref{eqn-211217-1002} with $\lambda\ge 0$, we have $U\in L^{x}_pL^{t}_\infty(\cQ_{R/4})$ and
	\begin{equation*}%\label{eqn-200918-0907}
		\|U\|_{L^{x}_pL^{t}_\infty(\cQ_{R/4})}\lesssim R^{-2/p}\|U\|_{L_p(\cQ_R)}.
	\end{equation*}
\end{lemma}
\begin{proof}
	Again, by Agmon's idea, we only deal with the case when $\lambda=0$.
	
	\textbf{The case when $p=2$.} The lemma follows from \cite[Proposition 4.1]{MR4387198}.	
	
	\textbf{The case when $p\in \big(2, \frac{2(m+2)}{m+1}\big)$.} Due to the time-independency, $u_t$ satisfies the same equation and boundary conditions. Hence by Lemma \ref{lem-211217-0940} $(a)$ with $u_t$ in place of $u$,
	\begin{equation}\label{eqn-200918-0606}
		\norm{Du_t}_{L_p(\cQ_{R/4})} \lesssim R^{-1}\norm{u_t}_{L_p(\cQ_{R/2})}.
	\end{equation}
	Testing \eqref{eqn-211217-1002} by $u_t|u_t|^{p-2}$ and then applying Young's inequality, where we need $p\ge 2$,
	we obtain
	\begin{equation}\label{eqn-200918-0607}
		\norm{u_t}_{L_p(\cQ_{R/2})}\lesssim R^{-1}\norm{Du}_{L_p(\cQ_R)}.
	\end{equation}
	In this process, standard techniques including the mollification and iteration arguments as in the proof of \cite[Lemma~4.3]{MR4387198} are needed. Here we omit the details. From \eqref{eqn-200918-0606} and \eqref{eqn-200918-0607}, we get
	\begin{equation}\label{eqn-201101-1007}
		\norm{Du_t}_{L_p(\cQ_{R/4})} \lesssim R^{-2}\norm{Du}_{L_p(\cQ_R)}.
	\end{equation}
	Now we use the Sobolev embedding (only in the $t$ variable) to obtain
	\begin{equation*}
		\norm{Du(\cdot,x)}_{L_\infty^t((-(R/4)^2,0))} \lesssim R^{2-2/p}\norm{Du_t(\cdot,x)}_{L_p((-(R/4)^2,0))}+R^{-2/p}\norm{Du(\cdot,x)}_{L_p((-(R/4)^2,0))}.
	\end{equation*}
	Taking the $L_p$ norm in $x$, and then using \eqref{eqn-201101-1007}, we have
	\begin{equation*}%\label{eqn-200923-0720}
		\norm{Du}_{L_p^xL_\infty^t(\cQ_{R/4})} \lesssim R^{-2/p}\norm{Du}_{L_p(\cQ_R)}.
	\end{equation*}
	
	\textbf{The case when $p\in \big(\frac{2(m+2)}{m+3},2\big)$}. In this case, we first see that $u \in \cH^1_2(\cQ_{R/2})$ by Lemma \ref{lem-211217-0940} $(b)$. From H\"older's inequality and the estimate with $p=2$, we obtain
	\begin{equation*}%\label{eqn-200918-0933}
		\norm{U}_{L_p^xL_\infty^t(\cQ_{R/4})} \lesssim R^{d/p-d/2}\norm{U}_{L_2^xL_\infty^t(\cQ_{R/4})} \lesssim R^{d/p-d/2-1} \norm{U}_{L_2{(\cQ_{R/2})}} \lesssim R^{-2/p}\norm{U}_{L_p{(\cQ_R)}}.
	\end{equation*}
	In the last inequality, we applied \eqref{eqn-200813-0849} with $\hat{p}=2$ and H\"older's inequality. The lemma is proved.
\end{proof}

%%========================================
\section{Higher regularity of \texorpdfstring{$\cH^1_p$}{H1p} solutions}	\label{sec-211227-0648}	
%%========================================
In this section, we prove the following regularity result by a level set argument.
\begin{proposition}\label{prop-200819-0953}
	Let $p\in \big(\frac{2(m+2)}{m+3},2\big]$ and $q\in (p, \infty)$.
	There exist constants $\theta,\gamma\in (0,1)$ and $\lambda_0\in (0, \infty)$ with
 $$
 (\theta, \gamma)=(\theta, \gamma)(d, \Lambda, M, p,q), \quad \lambda_0=\lambda_0(d, \Lambda, M, p, R_0)
 $$
such that if Assumptions \ref{A2} $(\theta)$ and \ref{A11} $(\gamma;m,M)$ are satisfied, then
	for any solution $u\in \cH^1_{p, \cD^T}(\cQ^T)$ to \eqref{220203@eq1}, 	where $\lambda \ge\lambda_0$ and $g_i,f\in L_p(\cQ^T)\cap L_{q,p}(\cQ^T)$, we have $u\in \cH^1_{q,p, \cD^T}(\cQ^T)$ satisfying
	\begin{equation}\label{eqn-201028-0508}
		\norm{U}_{L_{q,p}(\cQ^T)} \lesssim_{d,\Lambda, M, p,q}  \norm{F}_{L_{q,p}(\cQ^T)} +R_0^{2(1/q-1/p)} \norm{U}_{L_p(\cQ^T)}.
	\end{equation}
	When $\Gamma$ is time-independent, the same result is true for any $p\in \big(\frac{2(m+2)}{m+3}, \frac{2(m+2)}{m+1}\big)$.
\end{proposition}
The rest of Section \ref{sec-211227-0648} will be devoted to proving the proposition. Let us denote
\begin{equation*}
	\Phi_U(t):=\norm{U(t,\cdot)}_{L_p(\Omega)}\quad\text{and}\quad	\Phi_F(t):=\norm{F(t,\cdot)}_{L_p(\Omega)}.
\end{equation*}
%%========================================
\subsection{Decomposition}		\label{S6}
%%========================================
The key step in proving Proposition \ref{prop-200819-0953} is a suitable local decomposition of $\Phi_U(t)$.
\begin{lemma}\label{lem-201027-1030}
	Let $p\in \big(\frac{2(m+2)}{m+3},2\big]$ and $\hat{p}\in \big(2, \frac{2(m+2)}{m+1}\big)$.
	For the constants
	$$
	(\theta, \gamma)=(\theta, \gamma)(d,\Lambda,M, p), \quad \lambda_0=\lambda_0(d, \Lambda, M, p, R_0)
	$$
in Theorem \ref{MT1}, if Assumptions \ref{A2} $(\theta)$ and \ref{A11} $(\gamma;m,M)$ are satisfied, then the following assertion holds.
	For  any $t_0\in (-\infty, T]$, $R\in (0, R_0]$, and $u\in \cH^1_{p, \cD^T}(\cQ^T)$ satisfying \eqref{220203@eq1} with $\lambda\ge\lambda_0$ and $g_i, f\in L_p(\cQ^T)$, there exist nonnegative functions $\Phi_{W,R}(t)$ and $\Phi_{V,R}(t)$ defined on $(t_0-(R/16)^2,t_0)$,
	such that
	$$
	\Phi_U \le \Phi_{W,R} + \Phi_{V,R} \quad \text{in }\, (t_0-(R/16)^2,t_0),
	$$
	\begin{equation}\label{eqn-201027-1035-1}
		\int_{t_0-(R/16)^2}^{t_0}\Phi_{W,R}(t)^p\,dt \lesssim_{d,\Lambda,M, p, \hat{p}} (\theta+\gamma)^{\tau}\int_{t_0-R^2}^{t_0} \Phi_U(t)^p\,dt + \int_{t_0-R^2}^{t_0} \Phi_F(t)^p\,dt,
	\end{equation}
	where $\tau=p/2-p/\hat{p}$, and
	\begin{equation}\label{eqn-201027-1035-2}
		\sup_{t\in (t_0-(R/16)^2,t_0)}\Phi_{V,R}(t) \lesssim_{d,\Lambda,M, p, \hat{p}} \Bigg(\dashint_{t_0-R^2}^{t_0} \Phi_{U}(t)^p\,dt\Bigg)^{1/p} + \Bigg(\dashint_{t_0-R^2}^{t_0} \Phi_{F}(t)^p\,dt\Bigg)^{1/p}.
	\end{equation}
When $\Gamma$ is time-independent,
the same result holds  with $\tau=1-p/\hat{p}$ for any $p\in \big(\frac{2(m+2)}{m+3}, \frac{2(m+2)}{m+1}\big)$ and $\hat{p}\in \big(p, \frac{2(m+2)}{m+1}\big)$.
\end{lemma}

\begin{proof}
{\it Case 1: time-dependent $\Gamma$, $p\in  \big(\frac{2(m+2)}{m+3},2 \big]$, and $\hat{p}\in \big(2, \frac{2(m+2)}{m+1}\big)$.}
	
\textbf{The first decomposition: source terms.} By Theorem \ref{MT1}, there is a unique solution $u^{(1)}\in \cH^1_{p,\cD^T}(\cQ^T)$  to \eqref{220203@eq1} with $g_i\mathbb{I}_{(t_0-R^2,t_0)}$ and $f\mathbb{I}_{(t_0-R^2,t_0)}$ in place of $g_i$ and $f$, which satisfies
\begin{equation}\label{eqn-201006-0921}
\|U^{(1)}\|_{L_p(\cQ^T)} \lesssim \left(\int_{t_0-R^2}^{t_0}\Phi_F(t)^p\,dt\right)^{1/p}.
\end{equation}
Here $U^{(1)}=|Du^{(1)}|+\sqrt{\lambda}|u^{(1)}|$. Let $u^{(2)}:=u-u^{(1)}$.

Next, for any point $X_0=(x_0,t_0)$ with $x_0\in\overline{\Omega}$, since $u^{(2)}$ satisfies a homogeneous equation in $\cQ_R(X_0)$, by Lemma \ref{lem-211217-0940} $(b)$ we see that  $u^{(2)}\in \cH^1_{\hat{p}}(\cQ_{R/2}(X_0))$ and
\begin{equation}\label{eqn-211224-0857}
		(|U^{(2)}|^{\hat{p}})^{1/\hat p}_{\cQ_{R/2}(X_0)} \lesssim (U^{(2)})_{\cQ_{R}(X_0)},
\end{equation}
where $U^{(2)}=|Du^{(2)}|+\sqrt{\lambda} |u^{(2)}|$.
We claim that, for any $X_0=(t_0,x_0)$, we can find positive functions $W$ and $V$ satisfying
	\begin{equation}\label{eqn-220303-1037}
		U^{(2)} \leq W + V\quad \text{in}\,\,\cQ_{R/16}(X_0),
	\end{equation}
	\begin{equation}\label{eqn-201006-0919}
		\norm{W}_{L_p(\cQ_{R/16}(X_0))} \lesssim (\theta+\gamma)^{1/2-1/\hat{p}} \norm{U^{(2)}}_{L_p(\cQ_R(X_0))},
	\end{equation}
	and
	\begin{equation}\label{eqn-201006-0919-2}
		\norm{V}_{L_\infty^tL_p^x(\cQ_{R/16}(X_0))} \lesssim R^{-2/p}\norm{U^{(2)}}_{L_p(\cQ_R(X_0))}.
	\end{equation}
Let us first focus on the most complicated case -- the mixed case, i.e., when $X_0\in\Gamma$.

\textbf{The second decomposition: approximating $a_{ij}$ and $\Gamma$.}
In this case, we need to approximate $\Gamma$ by a time-independent separation and $a_{ij}$ by its average.  Take the coordinate system in Assumption \ref{A11} $(\gamma;m,M)$, and by translation, we may assume that $X_0=(0,0)$. Let $\chi=\chi(x)$ be the cut-off function on $\bR^d$ satisfying
\begin{align*}
	0\leq \chi\leq 1,\quad |D\chi|\lesssim_d \frac{1+M}{\gamma R},
	\\
	\chi=0 \quad \text{in }\, \{x:x^1<\gamma R, \, x^2>\phi-\gamma R\},
\\
\chi=1 \quad \text{in }\, \bR^d\setminus \{x:x^1<2\gamma R, \, x^2>\phi-2\gamma R\}.
\end{align*}
Then $\chi u^{(2)}$ satisfies
\begin{equation*}		%\label{201107@A1}
	\begin{cases}
		\cP_0 (\chi u^{(2)})-\lambda \chi u^{(2)} = D_i g^*_i+f^* & \text{in }\, \cQ_{R/2},\\
		\cB_0 (\chi u^{(2)}) = g^*_in_i& \text{on }\, (-(R/2)^2,0)\times N_{R/2},\\
		\chi u^{(2)} = 0 & \text{on }\, (-(R/2)^2,0)\times D_{R/2},
	\end{cases}
\end{equation*}
where
\begin{equation*}%\label{eqn-211220-0505-1}
	\cP_0 u := -u_t + D_i((a^{ij})_{\cQ_R} D_ju)\quad\text{and}\quad \mathcal{B}_0 u := (a^{ij})_{\cQ_R} D_j u n_i
\end{equation*}
have constant coefficients,
\begin{equation*}%\label{eqn-211220-0505-2}
	f^*=a^{ij}D_j u^{(2)}D_i\chi,\quad g_i^*=((a^{ij})_{\cQ_R}-a^{ij})D_j (\chi u^{(2)})+a^{ij}u^{(2)} D_j \chi,
\end{equation*}
and
\begin{equation*}%\label{eqn-220715-1217}
	\begin{split}
		D_{R/2} := \partial \Omega \cap B_{R/2}\cap  \{x: x^2>\phi-\gamma R\},
		\\
		N_{R/2} := \partial \Omega \cap B_{R/2}\cap  \{x: x^2<\phi-\gamma R\}.
	\end{split}
\end{equation*}
We first estimate $(1-\chi)u^{(2)}$.
By H\"older's inequality and our construction of $\chi$,
\begin{align}
			(|&(1-\chi)U^{(2)}|^p)^{1/p}_{\cQ_{R/4}} + (|D\chi u^{(2)}|^p)^{1/p}_{\cQ_{R/4}} \nonumber
			\\&\lesssim (|(1-\chi)U^{(2)}|^p)^{1/p}_{\cQ_{R/4}} + \frac{1}{\gamma R}(|\mathbb{I}_{\operatorname{supp}(D\chi)}u^{(2)}|^p)^{1/p}_{\cQ_{R/4}}\nonumber
		\\&\lesssim \gamma^{1/p-1/\hat{p}}\left((|U^{(2)}|^{\hat{p}})^{1/\hat{p}}_{\cQ_{R/4}} + \frac{1}{\gamma R}(|\mathbb{I}_{\operatorname{supp}(D\chi)}u^{(2)}|^{\hat p})^{1/\hat{p}}_{\cQ_{R/4}}\right)\nonumber
		\\&\lesssim \gamma^{1/p-1/\hat{p}} \left((|U^{(2)}|^{\hat{p}})^{1/\hat{p}}_{\cQ_{R/4}} + (|Du^{(2)}|^{\hat{p}})^{1/\hat{p}}_{\cQ_{R/2}}\right)  \label{eqn-211225-0729}
		\\&\lesssim \gamma^{1/p-1/\hat{p}}(|U^{(2)}|^p)^{1/p}_{\cQ_R}.\label{eqn-211217-1136-1}
\end{align}
Here, in \eqref{eqn-211225-0729}, we applied the boundary Poincar\'e inequality in \cite[Lemma~3.9]{MR4387198} on narrow regions, noting
\begin{equation*}
	\operatorname{dist}(\operatorname{supp}(D\chi)\cap \cQ_{R/4}, D_{R/2}\cap\cQ_{R/4}) < C\gamma R.
\end{equation*}
In \eqref{eqn-211217-1136-1}, we used \eqref{eqn-211224-0857} and H\"older's inequality.

\textbf{The third decomposition.}
Next we decompose $\chi u^{(2)}= u^{(3)} + u^{(4)}$ with
\begin{equation}\label{eqn-211220-0457}
\begin{cases}
\cP_0 u^{(3)}- \lambda u^{(3)} = D_i(g^*_i\mathbb{I}_{\cQ_{R/4}}) + f^*\mathbb{I}_{\cQ_{R/4}} &\quad \text{in }\, \cQ^0,\\
B_0 u^{(3)} = (g_i^*\mathbb{I}_{\cQ_{R/4}}) n_i&\quad \text{on }\, (-\infty,0)\times N_{R/2},\\
u^{(3)} = 0 &\quad \text{on }\, (-\infty,0)\times (\p\Omega\setminus N_{R/2}).
\end{cases}
\end{equation}
Notice that the new separation in the above problem is time-independent but may not satisfy  Assumption \ref{A11} $(b)$.
This is because the intersection of the boundary a small Reifenberg flat domain and a hyperplane might not be Reifenberg flat as the $x^1$-direction of the boundary (cf. Assumption \ref{A11}) at small scales might be almost parallelled to the normal direction of the hyperplane.
For such a reason, we apply \cite[Lemma~3.5]{MR4387198} which only requires the interfacial boundary to be time-independent to obtain the solution $u^{(3)}$ in $\cH^1_2(\cQ^0)$, whereas we are not able to utilize Theorem \ref{MT1} to get the solution in  $\cH^1_p(\cQ^0)$. This fact causes the restriction $p\le 2$.
 Clearly, $u^{(3)}=0$ for $t\leq -(R/4)^2$.
Moreover, we can test \eqref{eqn-211220-0457} by $u^{(3)}$ (with help of the usual Steklov average technique) to obtain
\begin{align}\label{eqn-211224-0851-1}
	(|U^{(3)}|^2)_{\cQ_{R/2}}
	\lesssim
	(|g_i^*|^2)_{\cQ_{R/4}}^{1/2}(|D_i u^{(3)}|^2)_{\cQ_{R/4}}^{1/2} + (|f^*u^{(3)}|)_{\cQ_{R/4}},
\end{align}
where
$$U^{(3)}:=|Du^{(3)}| + \sqrt{\lambda}|u^{(3)}|.$$
Furthermore, by H\"older's inequality, Assumption \ref{A2} $(\theta)$, and the boundary Poincar\'e inequality \cite[Lemma~3.9]{MR4387198} as above,
\begin{equation}\label{eqn-211224-0851-2}
	(|g_i^*|^2)_{\cQ_{R/4}}^{1/2} \lesssim (|u^{(2)}D\chi|^2)^{1/2}_{\cQ_{R/4}} + \theta^{1/2-1/\hat{p}}(|Du^{(2)}|^{\hat{p}})^{1/\hat{p}}_{\cQ_{R/4}} \lesssim (\theta+\gamma)^{1/2-1/\hat{p}}(|U^{(2)}|^{\hat p})^{1/{\hat p}}_{\cQ_{R/2}}.
\end{equation}
Similarly, by H\"older's inequality and \cite[Lemma~3.9]{MR4387198}, we have
\begin{align}
	(|f^*u^{(3)}|^2)_{\cQ_{R/4}}
	\nonumber
	&\lesssim (|\mathbb{I}_{\operatorname{supp}(D\chi)}Du^{(2)}|^2)^{1/2}_{\cQ_{R/4}} (|D\chi u^{(3)}|^2)^{1/2}_{\cQ_{R/4}}\\
	\label{eqn-211224-0851-3}
	& \lesssim \gamma^{1/2-1/\hat{p}}(|U^{(2)}|^{\hat p})^{1/{\hat p}}_{\cQ_{R/2}} (|U^{(3)}|^2)^{1/2}_{\cQ_{R/2}}.
\end{align}
From \eqref{eqn-211224-0851-1} -- \eqref{eqn-211224-0851-3}, canceling $(|U^{(3)}|^2)_{\cQ_{R/2}}^{1/2}$ from both sides, we have
\begin{equation*}
	(|U^{(3)}|^2)_{\cQ_{R/2}}^{1/2} \lesssim (\theta+\gamma)^{1/2-1/\hat{p}}(|U^{(2)}|^{\hat p})^{1/{\hat p}}_{\cQ_{R/2}}.
\end{equation*}
Recall that $p\le 2$. By H\"older's inequality, the inequality above, and \eqref{eqn-211224-0857}, we have
\begin{align}
	(|U^{(3)}|^p)^{1/p}_{\cQ_{R/2}}
	\nonumber
	&\lesssim (|U^{(3)}|^2)^{1/2}_{\cQ_{R/2}} \lesssim (\theta+\gamma)^{1/2-1/\hat{p}}(|U^{(2)}|^{\hat p})^{1/{\hat p}}_{\cQ_{R/2}}\\
	&\lesssim (\theta+\gamma)^{1/2-1/\hat{p}}(|U^{(2)}|^p)^{1/p}_{\cQ_{R}}.\label{eqn-201026-1123}
\end{align}
Let $$W:= |U^{(3)}| + |(1-\chi)U^{(2)}| + |D\chi u^{(2)}|.$$
From \eqref{eqn-211217-1136-1} and \eqref{eqn-201026-1123}, we obtain
\begin{equation}\label{eqn-220304-0230}
	\norm{W}_{L_p(\cQ_{R/4})} \lesssim (\theta+\gamma)^{1/2-1/\hat{p}} \norm{U^{(2)}}_{L_p(\cQ_R)}.
\end{equation}
Let $V:= |Du^{(4)}| + \sqrt{\lambda}|u^{(4)}|$, where $u^{(4)}=\chi u^{(2)} - u^{(3)}\in \cH^1_2(\cQ_R)$ satisfies
\begin{equation*}
\left\{
\begin{aligned}
\cP_0u^{(4)} - \lambda u^{(4)} = 0 &\quad \text{in }\, \cQ_{R/4},\\
B_0 u^{(4)} = 0 &\quad \text{on }\, (-(R/4)^2,0)\times N_{R/4},\\
u^{(4)}=0 &\quad \text{on }\, (-(R/4)^2,0)\times D_{R/4},
\end{aligned}
\right.
\end{equation*}
where
\begin{equation*}
	\begin{split}
		D_{R/4}:= \partial \Omega \cap B_{R/4}\cap  \{x: x^2>\phi-\gamma R\},
		\\
		N_{R/4}:= \partial \Omega \cap B_{R/4}\cap  \{x: x^2<\phi-\gamma R\}.
	\end{split}
\end{equation*}
By \cite[Proposition~4.1]{MR4387198} (noting that $|\p_t V|\leq |Du_t^{(4)}|+\sqrt{\lambda} |u_t^{(4)}|$) and H\"older's inequality, we have
\begin{align}
	\norm{\p_t V}_{L_p(\cQ_{R/8})}
\nonumber
&\lesssim R^{(d+2)/p-2-(d+2)}\norm{V}_{L_1(\cQ_{R/4})}\\
\label{eqn-201023-1123}
&\lesssim R^{-2}\norm{V}_{L_p(\cQ_{R/4})}.
\end{align}
Now we use the Sobolev embedding (in $t$) to obtain, for any $x\in \Omega_{R/16}$,
\begin{align*}
	\norm{V(\cdot,x)}_{L_\infty((-(R/8)^2,0))} &\lesssim R^{2-2/p}\norm{\p_t V(\cdot,x)}_{L_p((-(R/8)^2,0))} + R^{-2/p}\norm{V(\cdot,x)}_{L_p((-(R/8)^2,0))} .
\end{align*}
Taking the $L_p(\Omega_{R/8})$ norm in $x$, using \eqref{eqn-201023-1123}, and noting
$$
V\leq |D(\chi u^{(2)})| + \sqrt{\lambda}|\chi u^{(2)}| + |U^{(3)}|,
$$
we obtain
\begin{equation}\label{eqn-220304-0231}
	\begin{split}
		R^{2/p}\norm{V}_{L_\infty^tL_p^x(\cQ_{R/8})}
		&\lesssim R^{2/p}\norm{V}_{L_p^xL_\infty^t(\cQ_{R/8})}
		\lesssim \norm{V}_{L_p(\cQ_{R/4})}
		\\&\lesssim \norm{D(\chi u^{(2)})}_{L_p(\cQ_{R/4})} + \sqrt{\lambda}\norm{\chi u^{(2)}}_{L_p(\cQ_{R/4})} + \norm{U^{(3)}}_{L_p(\cQ_{R/4})}
		\\&\lesssim \norm{U^{(2)}}_{L_p(\cQ_{R})}.
	\end{split}
\end{equation}
In the last inequality, we used \eqref{eqn-211217-1136-1} for the estimate of $D\chi u^{(2)}$ and also \eqref{eqn-201026-1123}.
From the construction of $W$ and $V$, \eqref{eqn-220304-0230}, and \eqref{eqn-220304-0231}, we complete the construction satisfying \eqref{eqn-220303-1037} -- \eqref{eqn-201006-0919-2} when $X_0\in\Gamma$.
	
Next, for $X_0$ having one of the following three positions: $\cQ_R(X_0)\subset\cQ$ (interior), $\cQ_R(X_0)\cap\p\cQ\subset\cD$ (pure Dirichlet), or $\cQ_R(X_0)\cap\p\cQ\subset\cN$ (pure Neumann), the construction is similar. Actually it is simpler since no approximation of $\Gamma$ is presented. From these, the construction centered at any point $X_0$ can be achieved by a standard scaling and covering argument. The details are omitted.

Now we cover $(-(R/16)^2,0)\times \Omega$ with $\bigcup_{k}\cQ_{R/16}(X_0^{(k)})$  with the number of overlapping bounded by a number independent of $R$.
On each $\cQ_{R/16}(X_0^{(k)})$, we can define $W^{(k)}$ and $V^{(k)}$ as above. Then let
\begin{equation}\label{eqn-220219-0830}
	\Phi_{W,R}(t):=\bigg(\sum_k\bignorm{W^{(k)}(t,\cdot)\mathbb{I}_{\cQ_{R/16}(X_0^{(k)})}}^p_{L_p(\Omega)} + \bignorm{U^{(1)}(t,\cdot)}^p_{L_p(\Omega)}\bigg)^{1/p}
\end{equation}
and
\begin{equation}\label{eqn-220219-0833}
	\Phi_{V,R}(t):=\bigg(\sum_k\bignorm{V^{(k)}(t,\cdot)\mathbb{I}_{\cQ_{R/16}(X_0^{(k)})}}^p_{L_p(\Omega)}\bigg)^{1/p}.
\end{equation}
We immediately have $$\Phi_U(t) \leq \Phi_{W,R}(t) + \Phi_{V,R}(t)\quad \forall t\in(-(R/16)^2, 0).$$
Furthermore, from \eqref{eqn-220219-0830}, \eqref{eqn-201006-0919} with centers $X^{(k)}_0$, the fact $U^{(2)}\leq U + U^{(1)}$, and \eqref{eqn-201006-0921},
\begin{equation*}
\begin{split}
	&\int_{-(R/16)^2}^{0}\Phi_{W,R}(t)^p\,dt\\
&\lesssim (\theta+\gamma)^{p/2-p/\hat{p}}\int_{-R^2}^{0}
\norm{U^{(2)}(t,\cdot)}_{L_p(\Omega)}^p\,dt + \int_{-R^2}^{0} \Phi_F(t)^p\,dt\\
	&\lesssim (\theta+\gamma)^{p/2-p/\hat{p}}\int_{-R^2}^{0} \big(\norm{U^{(1)}(t,\cdot)}_{L_p(\Omega)}^p + \norm{U(t,\cdot)}_{L_p(\Omega)}^p{\big)}\,dt + \int_{-R^2}^{0} \Phi_F(t)^p\,dt\\
	&\lesssim (\theta+\gamma)^{p/2-p/\hat{p}}\int_{-R^2}^{0} \Phi_U(t)^p\,dt + \int_{-R^2}^{0} \Phi_F(t)^p\,dt.
\end{split}
\end{equation*}
This proves \eqref{eqn-201027-1035-1}. Similarly, \eqref{eqn-201027-1035-2} can be obtained from \eqref{eqn-220219-0833}, \eqref{eqn-201006-0919-2}, $U^{(2)}\leq U + U^{(1)}$, and \eqref{eqn-201006-0921}. This finishes the proof of the time-dependent case.

{\it Case 2: time-independent $\Gamma$, $p\in \big(\frac{2(m+2)}{m+3}, \frac{2(m+2)}{m+1}\big)$, and $\hat{p}\in \big(p, \frac{2(m+2)}{m+1}\big)$}.
We define $u^{(1)}$ as before and let $u^{(2)}=u-u^{(1)}$. For the local decomposition of $U^{(2)}$, in this case, we do not need to employ the cutoff argument with $\chi$.
Hence, the proof is simpler and the restriction $p\le 2$ is no longer needed.
  Let us give a sketch. We first freeze the coefficients by solving
$$
\begin{cases}
	\cP_0 u^{(3)}-\lambda u^{(3)} = D_i\big(((a^{ij})_{\cQ_R}-a^{ij})D_j u^{(2)}\mathbb{I}_{\cQ_{R/4}}\big) & \text{in }\, \cQ^0,\\
	\cB_0 u^{(3)} = ((a^{ij})_{\cQ_R}-a^{ij})D_j u^{(2)}\mathbb{I}_{\cQ_{R/4}}n_i& \text{on }\, \cN^0,\\
	u^{(3)} = 0 & \text{on }\, \cD^0.
\end{cases}
$$
By Theorem \ref{MT1}, the solution $u^{(3)}\in \cH^1_p(\cQ^0)$ exists and satisfies
\begin{equation}\label{eqn-220219-0858}
	\norm{U^{(3)}}_{L_p(\cQ_{R/4})} \lesssim \norm{(a^{ij})_{\cQ_R}-a^{ij})D_j u^{(2)}}_{L_p(\cQ_{R/4})} \lesssim (\theta+\gamma)^{1/p-1/\hat{p}} \norm{U^{(2)}}_{L_p(\cQ_R)}.
\end{equation}
In the last inequality, we use H\"older's inequality and the reverse H\"older's inequality as in \eqref{eqn-201026-1123}. We define
\begin{equation*}
	W:= U^{(3)}\quad\text{and}\quad V:=|Du^{(4)}|+\sqrt{\lambda}|Du^{(4)}|,
\end{equation*}
where $u^{(4)}:=u^{(2)}-u^{(3)}$. Since $v$ satisfies a homogeneous equation with time-independent separation satisfying the condition $(b)$ of Assumption \ref{A11} $(\gamma;m, M)$ in $\cQ_{R/4}$, by Lemma \ref{lem-200913-1026} and \eqref{eqn-220219-0858},
\begin{equation*}
	R^{2/p}\norm{V}_{L_\infty^tL_p^x(\cQ_{R/16})} \lesssim \norm{V}_{L_p(\cQ_{R/4})} \lesssim \norm{U^{(2)}}_{L_p(\cQ_{R/4})} + \norm{U^{(3)}}_{L_p(\cQ_{R/4})} \lesssim \norm{U^{(2)}}_{L_p(\cQ_R)}.
\end{equation*}
The rest of the proof remains the same as the time-dependent case.
\end{proof}

%========================================
\subsection{Level set estimates}
%========================================
In this section, we focus on the case when $\Gamma$ is time-dependent and $p\in \big(\frac{2(m+2)}{m+3},2\big]$. When $\Gamma$ is time-independent and $p\in \big(\frac{2(m+2)}{m+3}, \frac{2(m+2)}{m+1}\big)$, some minor changes are needed. See Section \ref{sec-220226-0919}.
For a function $h\in L_{1,{\rm loc}}(-\infty,T)$, we define its ($1$ dimensional) maximal function by
\begin{equation*}
	\mathcal{M}(h)(t_0) := \sup_{(a,b)\ni t_0}\dashint_a^b |h(t)|\mathbb{I}_{(-\infty,T)} \,dt.
\end{equation*}
We will estimate the following level sets
\begin{equation*}
	\mathcal{A}(s):=\big\{t\in(-\infty,T):\mathcal{M}(\Phi_U^p)(t)^{1/p}>s\big\},
\end{equation*}
\begin{equation}\label{eqn-220226-0925}
	\mathcal{B}(s):=\big\{t\in(-\infty,T):\mathcal{M}(\Phi_U^p)(t)^{1/p} + (\theta+\gamma)^{1/\hat{p}-1/2}\mathcal{M}(\Phi_F^p)(t)^{1/p}>s\big\}.
\end{equation}
\begin{lemma}\label{lem-211226-1120}
Let $p$, $\hat p$, and $u$ be as in Lemma \ref{lem-201027-1030}.
	There exists a constant $\kappa=\kappa(d,\lambda,M,p,\hat{p})>5$,
such that for any interval $(a,b)\subset(-\infty,T)$ with $|b-a|\leq (R_0/32)^2$ and
	\begin{equation}\label{eqn-211227-0707}
		|(a,b)\cap\mathcal{A}(\kappa s)|>(\theta+\gamma)^{p/2-p/\hat p}|b-a|,
	\end{equation}
we must have
\begin{equation*}
	(a,b)\subset\mathcal{B}(s).
\end{equation*}
\end{lemma}
\begin{proof}
	We prove by contradiction. Suppose that for some interval $(a,b)$ satisfying \eqref{eqn-211227-0707}, there exists some $t_1\in(a,b)\setminus \mathcal{B}(s)$, i.e.,
	\begin{equation}\label{eqn-211226-1020}
		\mathcal{M}(\Phi_U^p)(t_1)^{1/p} + (\theta+\gamma)^{1/\hat{p}-1/2}\mathcal{M}(\Phi_F^p)(t_1)^{1/p}\leq s.
	\end{equation}
Let
\begin{equation*}
	(a_1, b_1):=(a-|b-a|/2,\min\{b+|b-a|/2, T\}),\quad R=16\sqrt{|b_1-a_1|},
\end{equation*}
and observe that $(a_1, b_1)=(b_1-(R/16)^2, b_1)$.
By Lemma \ref{lem-201027-1030} with $t_0=b_{1}$ and such $R$, we have the decomposition
\begin{equation}\label{eqn-220215-1224}
	\Phi_U\leq \Phi_{W,R} + \Phi_{V,R}\quad\text{on}\,\,(a_1, b_1).
\end{equation}
Moreover, from \eqref{eqn-201027-1035-1}, \eqref{eqn-201027-1035-2}, \eqref{eqn-211226-1020}, and the fact that $t_1\in(a,b)\subset (b_1-R^2, b_1)$, we have
\begin{equation}\label{eqn-220215-1227}
\Bigg(\dashint_{a_1}^{b_1}
\Phi_{W,R}(t)^p\,dt\Bigg)^{1/p} \leq C_1(\theta+\gamma)^{1/2-1/\hat{p}}s \quad \text{and}\quad\sup_{t\in (a_1,b_1)}\Phi_{V,R}(t) \leq C_1s,
\end{equation}
where $C_1=C_1(d,\lambda,M,p,\hat{p})$.
Now for any $\tilde{t}\in (a,b)\cap\mathcal{A}(\kappa s)$, by definition, there exists an interval $(\tilde{a},\tilde{b})$ containing $\tilde{t}$, with $\tilde{b}\leq T$ and
	\begin{equation}\label{eqn-220215-0134}
		\dashint_{\tilde{a}}^{\tilde{b}}  \Phi_{U}(t)^p\,dt > (\kappa s)^p.
	\end{equation}
Actually, we must have $|\tilde{b}-\tilde{a}|\leq |b-a|/2$, since otherwise
\begin{equation*}
\begin{aligned}
	\mathcal{M}(\Phi_U^p)(t_1) &\geq \dashint_{\tilde{a}-|b-a|}^{\min\{\tilde{b}+|b-a|, T\}} \Phi_U(t)^p\,dt \\
	&\geq \frac{|\tilde{b}-\tilde{a}|}{2|b-a| + |\tilde{b}-\tilde{a}|}\,\dashint_{\tilde{a}}^{\tilde{b}}
\Phi_U(t)^p\,dt > \frac{1}{5}(\kappa s)^p > s^p
\end{aligned}
\end{equation*}
contradicting \eqref{eqn-211226-1020}.
Hence, $(\tilde{a},\tilde{b})\subset (a_1,b_1)$.
From this, \eqref{eqn-220215-1224} -- \eqref{eqn-220215-0134}, and the triangle inequality, we get for any $\tilde{t}\in(a,b)\cap\mathcal{A}(\kappa s)$,
	\begin{equation*}
	\begin{aligned}		\mathcal{M}\big(\Phi_{W,R}^p\mathbb{I}_{(a_1,b_1)}\big)
(\tilde{t})^{1/p}
		&\geq \Bigg(\dashint_{\tilde{a}}^{\tilde{b}}\Phi_{W,R}(t)^p\,dt\Bigg)^{1/p} \\
		&\geq \Bigg(\dashint_{\tilde{a}}^{\tilde{b}}\Phi_{U}(t)^p\,dt\Bigg)^{1/p} - \sup_{t\in(\tilde{a},\tilde{b})}\Phi_{V,R}(t) > \kappa s - C_1 s.
\end{aligned}
\end{equation*}
Hence, by the Hardy-Littlewood maximal function theorem and \eqref{eqn-220215-1227},
\begin{equation*}
\begin{aligned}
	|\mathcal{A}(\kappa s)\cap(a,b)|
	&\leq \big|\big\{\tilde t: \cM \big(\Phi_{W, R}^p \mathbb{I}_{(a_1, b_1)}\big)(\tilde{t})^{1/p}> \kappa s-C_1s\big\}\cap(a,b)\big| \\
	&\leq C(\kappa-C_1)^{-p}
C_1^p(\theta+\gamma)^{p/2-p/\hat{p}}|b-a|.
\end{aligned}
\end{equation*}
Here we also used the fact that $R^2\approx |b-a|$. Choosing $\kappa$ large enough, we reach a contradiction with \eqref{eqn-211227-0707}. The lemma is proved.
\end{proof}
From Lemma \ref{lem-211226-1120}, the Hardy-Littlewood maximal function theorem
\begin{equation*}
	|\mathcal{A}(\kappa s)|
\leq \|U\|^p_{L_p(\cQ^T)}/(\kappa s)^p,
\end{equation*}
and a measure theoretic lemma called ``crawling of the ink spot'' in \cite{MR0579490,MR0563790}, we have
\begin{corollary}\label{cor-211120-1226}
	For $\kappa$ in Lemma \ref{lem-211226-1120}, $s$ satisfying
	\begin{equation}\label{eqn-220219-1009-1}
		s\geq s_0:= \kappa^{-1}\|U\|_{L_p(\cQ^T)}\big((\theta+\gamma)^{p/2 - p/\hat p}R_0^2/32^2\big)^{-1/p},
	\end{equation}
and $\theta,\gamma$ satisfying
\begin{equation}\label{eqn-220219-1009-2}
	(\theta+\gamma)^{p/2-p/\hat{p}} < 1,
\end{equation}
	we have
	\begin{equation*}
		|\mathcal{A}(\kappa s)|\leq C (\theta+\gamma)^{p/2-p/\hat{p}}|\mathcal{B}(s)|.
	\end{equation*}
\end{corollary}
Here we omit the details, which can be found, for instance, in \cite[Proof of Lemma~5.4]{MR4387198}.
The key idea here is a stopping time argument: for any $t\in \mathcal{A}(\kappa s)$, we shrink the interval $(a,b)$ containing $t$ until the first time \eqref{eqn-211227-0707} holds. The condition \eqref{eqn-220219-1009-1} guarantees that we can start this procedure with $|b-a|=(R_0/32)^2$. The condition \eqref{eqn-220219-1009-2} together with the Lebesgue differentiation theorem guarantees that such procedure will stop.

%========================================
\subsection{Proof of Proposition \ref{prop-200819-0953}}\label{sec-220226-0919}
%========================================
\begin{proof}[Proof of Proposition \ref{prop-200819-0953}]
We mainly prove for general time-dependent $\Gamma$. For fixed $\hat{p}\in\big(p,\frac{2(m+2)}{m+1}\big)$,
let $\kappa$ and $s_0$ be the constants from Lemma \ref{lem-211226-1120} and Corollary \ref{cor-211120-1226}, respectively.
We also let $\theta$ and $\gamma$ be small numbers satisfying \eqref{eqn-220219-1009-2} to be chosen below.
For $S>s_0$,
	\begin{align}
		&\int_0^{\kappa S} |\mathcal{A}(s)|s^{q-1}\,ds = \kappa^{q} \int_0^{S} |\mathcal{A}(\kappa s)|s^{q-1}\,ds\nonumber
		\\&= \kappa^{q} \int_0^{s_0} |\mathcal{A}(\kappa s)|s^{q-1}\,ds + \kappa^{q} \int_{s_0}^S |\mathcal{A}(\kappa s)|s^{q-1}\,ds\nonumber
		\\&\leq \kappa^{q} \|U\|_{L_p(\cQ^T)}^p\int_0^{s_0} (\kappa s)^{-p}s^{q-1}\,ds + C\kappa^{	q}(\theta+\gamma)^{p/2-p/\hat{p}} \int_0^S |\mathcal{B}(s)|s^{q-1}\,ds.\label{eqn-211226-1132}
	\end{align}
Here in \eqref{eqn-211226-1132}, we applied the Chebyshev inequality and Corollary \ref{cor-211120-1226} for the two terms, respectively. Noting $q>p$ and
\begin{equation*}
	\mathcal{B}(s)\subset \mathcal{A}(s/2)\cup \{t:\mathcal{M}(\Phi_F^p)(t)>2^{-p}s^p(\theta+\gamma)^{p/2-p/\hat{p}}\},
\end{equation*}
using the integral formula for $L_p$ norms in terms of level sets and the Hardy-Littlewood maximal function theorem, we have
\begin{align*}
	&\int_0^{\kappa S} |\mathcal{A}(s)|s^{q-1}\,ds
	\leq C\kappa^{q-p}\|U\|_{L_p(\cQ^T)}^p s_0^{q-p} + C\kappa^{q}(\theta+\gamma)^{p/2-p/\hat{p}}\int_0^S|\mathcal{A}(s/2)|s^{q-1}\,ds
	\\&\quad + C\kappa^{q}(\theta+\gamma)^{p/2-p/\hat{p}}
\int_0^S\big|\big\{t:\mathcal{M}(\Phi_F^p)(t)>2^{-p}s^p(\theta+\gamma)^{p/2-p/\hat{p}}\big\}\big|s^{q-1}\,ds
	\\&\leq C_{\theta,\gamma,\kappa}
\big(\|U\|_{L_p(\cQ^T)}^q R_0^{2(p-q)/p} + \|\Phi_F\|_{L_q((-\infty, T))}^q\big) + C\kappa^{q}(\theta+\gamma)^{p/2-p/\hat{p}}
\int_0^{S/2} |\mathcal{A}(s)|s^{q-1}\,ds.
\end{align*}
Absorbing the integral involving $\mathcal{A}(s)$ on the right-hand side by choosing $\theta$ and $\gamma$ small enough, passing $S\rightarrow \infty$, we reach the desired estimate.

When $\Gamma$ is time-independent, the proof is almost the same if we change the definition of $\mathcal{B}(s)$ in \eqref{eqn-220226-0925} to
\begin{equation*}
	\mathcal{B}(s):=\big\{t\in(-\infty,T):\mathcal{M}(\Phi_U^p)(t)^{1/p} + (\theta+\gamma)^{1/\hat{p}-1/p}\mathcal{M}(\Phi_F^p)(t)^{1/p}>s\big\}.
\end{equation*}
The details are omitted.
\end{proof}

%========================================
\section{Proof of Theorem \ref{thm-201022-0843}}		\label{S5}
%========================================
With Proposition \ref{prop-200819-0953} at hand, now we prove Theorem \ref{thm-201022-0843}.
\begin{proof}[Proof of Theorem \ref{thm-201022-0843}]
We consider the following three cases.

{\it Case 1: time-dependent $\Gamma$, $p\leq 2$}. By approximation, we may assume that $f$ and $g_i$ have compact support in time, and hence $f, g_i\in L_q^tL_p^x\subset L_p^{t,x}$. By Theorem \ref{MT1}, we can find a solution $u\in \cH^1_p$. Furthermore, by Proposition \ref{prop-200819-0953}, $u\in \cH^1_{q,p}$. To show \eqref{eqn-201028-0521}, we are left with absorbing the $U$ term on the right-hand side of \eqref{eqn-201028-0508}. This step is standard, which can be done by multiplying a cut-off function in the $t$ variable with sufficiently small support, using H\"older's inequality, and then choosing $\lambda$ large enough. Such argument can be found in the proof of \cite[Corollary~5.2]{MR4261267}. The range of $(p,q)$ corresponds to the shaded trapezoid area in Figure \ref{fig:1/p-1/q}.

{\it Case 2: time-dependent $\Gamma$, $p> 2$}. In this case, we interpolate the $\cH^1_{\tilde{q},2}$ and $\cH^1_{\tilde{p}}$ results, where $\tilde{q}>2$ can be sufficiently large and $\tilde{p}\in (2, 2(m+2)/(m+1))$. To be more precise, let $\vartheta\in(0,1)$ be the number such that
$$\frac{1}{p}= \frac{\vartheta}{2} + \frac{1-\vartheta}{\tilde{p}}\quad\text{and}\quad \frac{1}{q}= \frac{\vartheta}{\tilde{q}} + \frac{1-\vartheta}{\tilde{p}}.$$
Then the $\cH^1_{q,p}$ solvability can be obtained from the $\cH^1_{\tilde{p}}$ and $\cH^1_{\tilde{q},2}$ results by applying the Riesz-Thorin interpolation theorem. Here we also used the following fact which can be found, for instance, in \cite[Theorem~1.18.1]{MR0500580}:
\begin{equation*}
	[L^t_{\tilde{q}}(L^x_2), L^t_{\tilde{p}}(L^x_{\tilde{p}})]_\vartheta = L^t_{q}([L^x_2,L^x_{\tilde{p}}]_\vartheta),
\end{equation*}
where $[\, \cdot\, , \, \cdot\, ]_\vartheta$ represents the complex interpolation space. The range of $(p,q)$ corresponds to the shaded triangle area in Figure \ref{fig:1/p-1/q}.

{\it Case 3: time-independent $\Gamma$, $p\in \big(\frac{2(m+2)}{m+3}, \frac{2(m+2)}{m+1}\big)$.} For $q>p$, the proof is exactly the same as the first case by using the last assertion of Proposition \ref{prop-200819-0953}. For $q<p$, the result can be obtained by duality.

This finishes the proof of Theorem \ref{thm-201022-0843}.
\end{proof}

\section*{Acknowledgments}
	The authors would like to thank the anonymous referee for careful reading of the manuscript and helpful comments.

\section*{Data availability statements}
Data sharing not applicable to this article as no datasets were generated or analysed during the current study.

\bibliographystyle{plain}

\end{document}